\documentclass[a4paper,12pt,final]{amsart}
\usepackage{times,a4wide,mathrsfs,amssymb,amsmath,amsthm,xypic}

\newcommand{\C}{\mathbb{C}}
\newcommand{\ZZ}{\mathbb{Z}}

\newcommand{\LLL}{\mathbb{L}}
\newcommand{\QQ}{\mathbb{Q}}
\newcommand{\NN}{\mathbb{N}}
\newcommand{\PP}{\mathbb{P}}

\newcommand{\Sy}{\mathfrak S}

\newcommand{\ZZZ}{\mathcal Z}

\newcommand{\MM}{\mathcal M}

\newcommand{\FF}{\mathcal F}

\newcommand{\wt}{\widetilde}
\newcommand{\rom}{\romannumeral}

\newcommand{\llsvs}{\textup{LLSvS}}
\newcommand{\splitting}{decomposition}

\DeclareMathOperator{\aut}{Aut}

\DeclareMathOperator{\ide}{id}

\DeclareMathOperator{\ima}{Im}

\DeclareMathOperator{\sym}{Sym}
\DeclareMathOperator{\NS}{NS}
\newtheorem{theorem}{Theorem}[section]
\newtheorem{claim}[theorem]{Claim}
\newtheorem{lemma}[theorem]{Lemma}

\newtheorem{corollary}[theorem]{Corollary}
\newtheorem{proposition}[theorem]{Proposition}

\newtheorem{remark}[theorem]{Remark}
\newtheorem{definition}[theorem]{Definition}
\newtheorem{convention}{Conventions}

\newtheorem{nonumberingt}{Acknowledgements}

\begin{document}

\author[Chiara Camere]
{Chiara Camere}

\address{Chiara Camere, Dipartimento di Matematica "F. Enriques",
Universit\`a degli Studi di Milano, Via Cesare Saldini 50, 20133 Milano (MI) } 
\email{chiara.camere@unimi.it}

\author[Alberto Cattaneo]
{Alberto Cattaneo}

\address{Alberto Cattaneo, Mathematisches Institut and Hausdorff Center for Mathematics, Universit\"at Bonn, Endenicher Allee 60, 53115 Bonn, Germany. }
\email{cattaneo@math.uni-bonn.de}

\author[Robert Laterveer]
{Robert Laterveer}

\address{Robert Laterveer, Institut de Recherche Math\'ematique Avanc\'ee,
CNRS -- Universit\'e 
de Strasbourg,\
7 Rue Ren\'e Des\-car\-tes, 67084 Strasbourg CEDEX,
France.}
\email{robert.laterveer@math.unistra.fr}

\title[On the Chow ring of certain LLS\MakeLowercase{v}S eightfolds]{On the Chow ring of certain Lehn--Lehn--Sorger--van Straten eightfolds}

\begin{abstract} We consider a $10$-dimensional family of Lehn--Lehn--Sorger--van Straten hyperk\"ahler eightfolds which have a non-symplectic automorphism of order $3$. Using the theory of finite-dimensional motives, we show that the action of this automorphism on the Chow group of $0$-cycles is as predicted by the Bloch--Beilinson conjectures. We prove a similar statement for the anti-symplectic involution on varieties in this family. This has interesting consequences for the intersection product of the Chow ring of these varieties.
\end{abstract}


\maketitle

\section{Introduction} \label{sec: intro}

Let $X$ be a smooth complex projective variety. We denote by $A^i(X):=CH^i(X)_{\QQ}$ the Chow groups of $X$, i.e.\ the groups of codimension-$i$ algebraic cycles on $X$ with coefficients in $\QQ$, modulo rational equivalence.   
Chow groups encode a lot of the geometry of a variety, but their structure is far from being well understood: Bloch's conjecture predicts that on any smooth projective variety $X$ of dimension $n$, for any correspondence $\Gamma\in A^n(X\times X)$ such that $\Gamma_*=0$ on $H^{0,i}(X)$ for all $i>0$, then $\Gamma_*:A^n_{hom}(X)\rightarrow A^n(X)$ is trivial (see \cite{B}).

Hyperk\"ahler varieties, defined as smooth compact complex projective manifolds $X$ which are simply connected and carry a holomorphic symplectic two-form $\omega_X$ such that $H^{2,0}(X)=\C\omega_X$, are conjectured to have Chow groups with an even richer structure, by the ``(weak) splitting property conjecture'' of Beauville \cite{Beau3}.

The aim of this paper is to investigate Chow groups of some special Lehn--Lehn--Sorger--van Straten (or \llsvs{}) hyperk\"ahler eightfolds $Z(Y)$ (see \cite{LLSS}), and more precisely those associated to a cyclic cubic fourfold not containing a plane $Y\subset\PP^5$. The eightfold $Z(Y)$ and the $4$-dimensional Fano variety of lines $F(Y)$ are explicit projective models of hyperk\"ahler varieties. As Voisin observed in \cite{V14}, they are related by a rational map of degree six 
  \[ \psi: F(Y)\times F(Y)\dashrightarrow Z(Y)\ \]
  \noindent whose construction will be recalled in Subsection \ref{ss:psi}. 

As known from unpublished work of Lehn, Lehn, Sorger and van Straten, the \llsvs{} eightfold $Z(Y)$ is equipped with an anti-symplectic involution $\iota \in \aut(Z(Y))$ (see \cite{MLe}). This automorphism is induced, via the rational map $\psi: F(Y) \times F(Y) \dashrightarrow Z(Y)$, by the involution of $F(Y) \times F(Y)$ exchanging the two factors (Equation (\ref{rmk: compatibility involution})).  
  
The family of cyclic cubic fourfolds $Y : f(x_0,\ldots,x_4) + x_5^3=0$, with $f$ a homogeneous polynomial of degree three, is ten-dimensional and the very general element in it does not contain a plane, as observed in \cite[Remark 6.4]{CC} (and also in \cite[Proposition 6.5]{Pan}). Let $\sigma \in \text{PGL}(6)$ be the automorphism:
\begin{equation} \label{eq: sigma}
\sigma(x_0: \ldots: x_5) = (x_0: \ldots : x_4 : \xi x_5) \ ,
\end{equation}
\noindent where we set $\xi := \exp(\frac{2 \pi i}{3})$. For $Y$ cyclic we have $\sigma(Y) = Y$, therefore $\sigma\vert_{Y}$ is an automorphism of $Y$. By \cite[Section 6.1]{CC}, $\sigma$ induces a non-symplectic automorphism $g \in \textrm{Aut}(Z(Y))$ which is of order $3$. Moreover, if $\mathcal{Y}$ is the family of cyclic cubic fourfolds not containing a plane, then the deformation family of hyperk\"ahler eightfolds $\left\{ Z(Y) \mid Y \in \mathcal{Y} \right\}$ is still $10$-dimensional (see \cite[Theorem 6.8]{CC}).

We investigate the action of the two non-symplectic automorphisms $\iota$ and $g$ on the Chow group of $0$-cycles of hyperk\"ahler eightfolds in this $10$-dimensional family and we partially show that it is as predicted by the Bloch--Beilinson conjectures. In Sections \ref{section - Preliminaries} and \ref{ss:vois}, after recalling all the definitions and the constructions which we will use throughout the paper, we start our analysis, based on previous work by the third author on Chow groups of Fano varieties associated to cyclic cubic fourfolds \cite{nonsymp3}, together with fundamental work on the Chow ring of Fano varieties of cubic fourfolds by Shen and Vial \cite{SV}.  Using Voisin's rational map $\psi$ and finite-dimensionality of $F(Y)$ (in the sense of Kimura \cite{Kim}), we construct in Proposition \ref{prop:split} a \splitting{} 
  \[ A^8(Z(Y))=\bigoplus_j A^8_{(j)}(Z(Y))\] 
  induced by certain projectors. This \splitting{} is in fact also compatible with Voisin's descending filtration $S_*$ (see \cite{V14}), as shown in Corollary \ref{cor:split}.
  
In Section \ref{section - Main results} we prove the main results of the paper. The first one is as follows:

  \begin{theorem} \label{main}
   Let $Y\subset\PP^5$ be a cyclic cubic fourfold not containing a plane and let $Z=Z(Y)$ be the associated \llsvs{} eightfold. Let $g\in\aut(Z)$ be the automorphism of order $3$ induced by the automorphism (\ref{eq: sigma}) of $Y$. Then

  \[  g_\ast + g^2_\ast +g^3_\ast=3\ide\colon\ \ \   A^8_{(0)}(Z)  \ \to\ A^8_{}(Z)  \ ,\]   
  whereas
  \[  g_\ast + g^2_\ast +g^3_\ast=0\colon\ \ \   A^8_{(j)}(Z)  \ \to\ A^8_{}(Z)   \ \ \ \hbox{for}\ j\in\{2,4\}\ .\]   
  \end{theorem}
  
  The cases $j=6,8$ are still open because of the difficulty of constructing projectors which isolate the $g^*$-invariant part of $\sym^{\frac{j}{2}}H^2_{tr}(Z)$.
  
  The next result concerns the anti-symplectic involution present on any \llsvs{} eightfold:
   
   \begin{theorem} \label{main2}
    Let $Y\subset\PP^5$ be a Pfaffian cyclic cubic fourfold not containing a plane and let $Z=Z(Y)$ be the associated \llsvs{} eightfold. Then the involution $\iota\in\aut(Z)$ satisfies:
  \[ \iota_\ast=\ide\colon\ \ A^8_{(j)}(Z)\ \to\ A^8_{}(Z)\ \ \ \hbox{for}\ j\in\{0,4,8\}\ ,\]
  whereas
  \[  \iota_\ast =-\ide\colon\ \ \   A^8_{(j)}(Z)  \ \to\ A^8_{}(Z)   \ \ \ \hbox{for}\ j\in\{2,6\}\ .\]   
  \end{theorem}
  
  Theorem \ref{main2} is expected to hold true for {\em every\/} \llsvs{} eightfold; one problem in proving such a statement is that one would first need to make sense of the pieces $A^8_{(j)}(Z)$ (cf.\ Remark \ref{general}).

As a consequence of our main results, in Section \ref{section - Consequences} we describe some constant cycles subvarieties on a \llsvs{} eightfold $Z$.

Given a projective quotient variety, defined as the quotient $Q=X/G$ of a smooth projective variety $X$ by a finite subgroup $G \subset \aut(X)$, a result by Fulton \cite[Example 17.4.10]{F} implies that $Q$ inherits unaltered the formalism of correspondences (as remarked in \cite[Example 16.1.13]{F}). This allows us to consider motives $(Q,p,0)\in\MM_{\rm rat}$, where  $p\in A^n(Q\times Q)$ is a projector. Moreover, as a consequence of Manin's identity principle, in $\MM_{\rm rat}$ we have an isomorphism 
  $  h(Q)\cong h(X)^G:=(X,\Delta_G,0)$, where $\Delta_G:={1\over \vert G\vert}{\sum_{g\in G}}\Gamma_g$.

Another consequence of our main results are some corollaries concerning the intersection product in the Chow ring of the quotient varieties of $Z$.

\begin{corollary} \label{cor5} Let $Y\subset\PP^5$ be a Pfaffian cyclic cubic fourfold not containing a plane and let $Z=Z(Y)$ be the associated \llsvs{} eightfold. Let $Q:=Z/\langle g \rangle$ be the quotient. Then
   \[ \begin{split} &\ima \bigl( A^4(Q)\otimes A^2(Q)\otimes A^2(Q)\ \to\ A^8(Q)\bigr)\\
                     = & \ima \bigl( A^3(Q)\otimes A^3(Q)\otimes A^2(Q)\ \to\ A^8(Q)\bigr)\\
                      = & \ima \bigl( A^5(Q)\otimes A^2(Q)\otimes A^1(Q)\ \to\ A^8(Q)\bigr)\\
                      \cong & \QQ\ .\\
                     \end{split}\]
\end{corollary}   

 \begin{corollary} \label{cor6} Let $Y\subset\PP^5$ be a Pfaffian cyclic cubic fourfold not containing a plane and let $Z=Z(Y)$ be the associated \llsvs{} eightfold. For the anti-symplectic involution $\iota\in\aut(Z)$ define $T := Z/\langle\iota\rangle$. Then
   \[ \begin{split} 
                      & \ima \bigl( A^2(T)^{}\otimes    A^2(T)^{}\otimes  A^2(T)^{}\otimes A^2(T)^{}  \ \to\ A^8(T)\bigr)\\
                      = & \ima \bigl( A^3(T)^{}\otimes A^2(T)^{}\otimes A^2(T)^{}\otimes A^1(T)^{}\ \to\ A^8(T)\bigr)\\
                      \cong & \QQ\ .\\
                     \end{split}\]
\end{corollary}            
      
This is similar to the behaviour of the Chow ring of $K3$ surfaces \cite{BV} and of Calabi--Yau complete intersections \cite{V13}, \cite{LFu}. Presumably, Corollary \ref{cor5} holds for {\em all\/} cyclic cubics, and Corollary \ref{cor6} holds for {\em all\/} \llsvs{} eightfolds (this would follow from Beauville's ``splitting property conjecture'' \cite{Beau3}). It would be interesting to prove this. In our argument, we need the Pfaffian hypothesis in order to have a bigrading on the full Chow ring of $Z$ (cf.\ Remark \ref{pfaffian}), while we need the cyclic hypothesis in order to have a finite-dimensional motive.

Also, it would be interesting to prove our results for the eightfold of $K3^{[4]}$-type constructed by Ouchi \cite{Ouchi}, which extends the \llsvs{} construction to cubic fourfolds containing a plane.

 \vskip0.6cm

\begin{convention} Throughout the paper, the word \emph{variety} will be used to refer to a reduced, irreducible separated scheme of finite type over $\C$. A reduced subscheme of a variety is a \emph{subvariety} if all of its irreducible components have the same dimension. If $X$ is a variety, we define $H^j(X)$ to be its singular $j$-th cohomology group with rational coefficients. 

We will always consider Chow groups with coefficients in $\QQ$. The Chow group of cycles of dimension $j$ on $X$, with $\QQ$-coefficients, will be denoted by $A_j(X)$. If $X$ is a smooth $n$-dimensional variety, the notations $A_j(X)$ and $A^{n-j}(X)$ are equivalent. 

We define $A^j_{hom}(X)$, $A^j_{AJ}(X)$ to be the subgroups of homologically trivial, respectively Abel--Jacobi trivial, cycles of codimension $j$. For $\sigma\in\aut(X)$, we will use the notation $A^j(X)^\sigma$ (respectively $H^j(X)^\sigma$) for the subgroup of $A^j(X)$ (respectively $H^j(X)$) which is invariant with respect to the action of $\sigma^*$.

If $f\colon X\to Y$ is a morphism of varieties, we will denote by $\Gamma_f\in A_\ast(X\times Y)$ the graph of $f$.

The notation $\MM_{\rm rat}$ will refer to the (contravariant) category of Chow motives, i.e.\ pure motives with respect to rational equivalence (see \cite{MNP}, \cite{Sc}).

\end{convention}

\section{Preliminaries}\label{section - Preliminaries}

\subsection{Finite-dimensional motives}

We recall that the objects in $\MM_{\rm rat}$ are motives, i.e.\ pairs $M=(X,p)$ consisting of a smooth projective variety $X$ and of a projector $p\in A^*(X\times X)$ with $p^2=p$. Given $M=(X,p)$ and $N=(Y,q)$ in $\MM_{\rm rat}$, by definition ${\rm Hom}(M,N)=q\circ A^*(X\times Y)\circ p$, where composition is intended as the composition of correspondences: if $f\in A^*(X\times Y), g\in A^*(Y\times Z)$, then $g\circ f:=(p_{13})_*(p_{12}^*f\cdot p_{23}^*g)\in A^*(X\times Z)$. In particular, for a smooth projective variety $X$ we will consider the motive $h(X):=(X,\Delta)$ in $\MM_{\rm rat}$.

The definition of finite-dimensional motive, in the sense of Kimura, can be found in \cite{An}, \cite{J4}, \cite{Kim}, \cite{MNP}. 
For varieties with finite-dimensional motive, the following nilpotence theorem holds:

\begin{theorem}[Kimura \cite{Kim}]\label{nilp} Let $X$ be a smooth projective variety of dimension $n$ with finite-dimensional motive. Let $\Gamma\in A^n(X\times X)_{}$ be a correspondence which is numerically trivial. Then there is $N\in\NN$ such that
     \[ \Gamma^{\circ N}=0\ \in A^n(X\times X)_{}\ .\]
\end{theorem}

We point out that the nilpotence property (for all powers of $X$) can be used to provide an alternative definition of finite-dimensional motive (see \cite[Corollary 3.9]{J4}). It is conjectured (cf.\ \cite{Kim}) that all varieties have finite-dimensional motive: even though this is still far from being proven, several non-trivial examples are known.

We will be using the following result from \cite[Theorem 3.1]{fam} and \cite[Theorem 4]{fano}:

\begin{theorem}\label{findim} Let $Y\subset\PP^5$ be a smooth cubic fourfold given by an equation
\[ f(x_0,\ldots,x_4) + x_5^3=0\ \]
\noindent (henceforth, we will call cubic fourfolds of this type {\em cyclic cubics}). Let $F=F(Y)$ be the Fano variety of lines on $Y$. Then $Y$ and $F$ have finite-dimensional motive.  
\end{theorem}

\subsection{CK decomposition}

\begin{definition}[Murre, Vial \cite{Mur}, \cite{V3}] \leavevmode

\noindent
(\rom1)
A smooth projective variety $X$ of dimension $n$ is said to have a {\em Chow--K\"unneth (CK) decomposition\/} if the diagonal in $A^n(X\times X)$ admits a decomposition
   \[ \Delta_X= \pi^0_X+ \pi^1_X+\cdots +\pi^{2n}_X\ \ \ \hbox{in}\ A^n(X\times X)\ , \]
  where the $\pi^i_X$'s are mutually orthogonal idempotents such that $(\pi^i_X)_\ast H^\ast(X)= H^i(X)$.

  \noindent
  (\rom2)
  A motive $M=(X,p,0)\in \MM_{\rm rat}$ has a {\em CK decomposition\/} if there is a decomposition
  \[ p=  \pi^0 + \pi^1 +\cdots + \pi^{2n}  \ \ \ \hbox{in}\ A^n(X\times X)\ , \]
  where the $\pi^i$'s are mutually orthogonal idempotents such that $(\pi^i)_\ast H^\ast(M)= H^i(M)$.  
  \end{definition}

\begin{remark} The existence of a CK decomposition for any smooth projective variety is part of Murre's conjectures \cite{Mur}, \cite{J4}. The extension of this notion from varieties to motives is made in \cite[Conjecture 2.1]{V3}.
\end{remark}  

The following provides a practical way of constructing CK decompositions.

\begin{proposition} \label{jannsen} Let $M=(X,p,0)\in \MM_{\rm rat}$. Assume that $M$ is finite-dimensional and that $X$ verifies the Lefschetz standard conjecture. Then $M$ admits a CK decomposition.
\end{proposition}

\begin{proof} This follows from \cite[Lemma 5.4]{J2}.
\end{proof}

\subsection{MCK decomposition}

\begin{definition}[Shen--Vial \cite{SV}] Let $X$ be a smooth projective variety of dimension $n$ and $\Delta_X^{sm}\in A^{2n}(X\times X\times X)$ be the class of the small diagonal
  \[ \Delta_X^{sm}:=\bigl\{ (x,x,x)\ \vert\ x\in X\bigr\}\ \subset\ X\times X\times X\ .\] 
  
  A {\em multiplicative Chow--K\"unneth (MCK) decomposition\/} is a CK decomposition $\{\pi_X^i\}$ of $X$ which is {\em multiplicative\/}, i.e.\ it satisfies
  \[ \pi_X^k\circ \Delta_X^{sm}\circ (\pi_X^i\times \pi_X^j)=0\ \ \ \hbox{in}\ A^{2n}(X\times X\times X)\ \ \ \hbox{for\ all\ }i+j\not=k\ .\]

 A {\em weak MCK decomposition\/} is a CK decomposition $\{\pi_X^i\}$ of $X$ such that:
    \[ \Bigl(\pi_X^k\circ \Delta_X^{sm}\circ (\pi_X^i\times \pi_X^j)\Bigr){}_\ast (a\times b)=0 \ \ \ \hbox{for\ all\ } a,b\in A^\ast(X)\ .\]
  \end{definition}
  
  \begin{remark} As a correspondence from $X\times X$ to $X$, the small diagonal induces the following {\em multiplication morphism\/}:
    \[ \Delta_X^{sm}\colon\ \  h(X)\otimes h(X)\ \to\ h(X)\ \ \ \hbox{in}\ \MM_{\rm rat}\ .\]
 Suppose now that $X$ has a CK decomposition
  \[ h(X)=\bigoplus_{i=0}^{2n} h^i(X)\ \ \ \hbox{in}\ \MM_{\rm rat}\ .\]
  This decomposition is multiplicative if the composition
  \[ h^i(X)\otimes h^j(X)\ \to\ h(X)\otimes h(X)\ \xrightarrow{\Delta_X^{sm}}\ h(X)\ \ \ \hbox{in}\ \MM_{\rm rat}\]
  factors through $h^{i+j}(X)$, for any $i,j$.
  
  Assuming that $X$ has a weak MCK decomposition, if we set
    \[ A^i_{(j)}(X):= (\pi_X^{2i-j})_\ast A^i(X)\]
    we obtain a bigraded ring structure on the Chow ring, because the intersection product maps $A^i_{(j)}(X)\otimes A^{i^\prime}_{(j^\prime)}(X) $ to  $A^{i+i^\prime}_{(j+j^\prime)}(X)$.

If $X$ has an MCK decomposition, it is expected (yet not proven) that the following equalities hold:
\[ A^i_{(j)}(X)\stackrel{??}{=}0\ \ \ \hbox{for}\ j<0\ ,\ \ \ A^i_{(0)}(X)\cap A^i_{hom}(X)\stackrel{??}{=}0.\]
This is related to Murre's conjectures B and D, which were formulated for any CK decomposition \cite{Mur}.
      
  The property of having an MCK decomposition is severely restrictive, and it is closely related to Beauville's ``(weak) splitting property'' \cite{Beau3}. For a wider discussion, and examples of varieties with an MCK decomposition, we refer to \cite[Section 8]{SV}, as well as to  \cite{FTV}, \cite{LV}, \cite{SV2}, \cite{V6}.
    \end{remark}

In what follows, we will make use of the following result.

\begin{theorem}[Shen--Vial \cite{SV}]\label{sv} Let $Y\subset\PP^5(\C)$ be a smooth cubic fourfold and let $F:=F(Y)$ be the Fano variety of lines on $Y$. There exists a CK decomposition $\{\Pi_F^i\}$ for $F$ and 
  \[ (\Pi_F^{2i-j})_\ast A^i(F) = A^i_{(j)}(F)\ ,\]
  where the right-hand side denotes the \splitting{} of the Chow groups defined in terms of the Fourier transform as in \cite[Theorem 2]{SV}. Moreover, we have
  \[ A^i_{(j)}(F)=0\ \ \ \hbox{if\ }j<0\ \hbox{or\ }j>i\ \hbox{or\ } j\ \hbox{is\ odd}\ .\]
  
  For $Y$ very general, the Fourier decomposition $A^\ast_{(\ast)}(F)$ forms a bigraded ring, hence
  $\{\Pi_F^i\}$ is a weak MCK decomposition.
    \end{theorem}

\begin{proof} (A remark on notation: what we denote by $A^i_{(j)}(F)$ is denoted by $CH^i(F)_{(j)}$ in \cite{SV}).

The existence of a CK decomposition $\{\Pi_F^i\}$ is \cite[Theorem 3.3]{SV}, combined with the results of \cite[Section 3]{SV} to ensure that \cite[Theorem 2.2 and 2.4]{SV} are satisfied. According to \cite[Theorem 3.3]{SV}, the given CK decomposition agrees with the Fourier decomposition of the Chow groups. The ``moreover'' part is because the $\Pi_F^i$'s are shown to satisfy Murre's conjecture B \cite[Theorem 3.3]{SV}. The statement for very general cubic fourfolds is \cite[Introduction, Theorem 3]{SV}.
    \end{proof}

\begin{remark}\label{pity} Unfortunately, it is not yet known that the Fourier decomposition of \cite{SV} induces a bigraded ring structure on the Chow ring for {\em all\/} Fano varieties of smooth cubic fourfolds. For one thing, it has not yet been proven that $A^2_{(0)}(F)\cdot A^2_{(0)}(F)\subset A^4_{(0)}(F)$ (cf.\ \cite[Section 22.3]{SV} for discussion, and \cite[Proposition A.7]{FLV} for some recent progress).
\end{remark}

\begin{theorem}[Vial {\cite[Theorem 1]{V6}}]\label{K3m} Let $X$ be a hyperk\"ahler variety which is birational to a Hilbert scheme $S^{[m]}$, where $S$ is a $K3$ surface. Then $X$ has an MCK decomposition $\lbrace \pi_X^i\rbrace$.
\end{theorem}

  \subsection{Refined CK decomposition}

In the following, we will use the term \emph{manifolds of $K3^{[m]}$--type} to refer to hyperk\"ahler manifolds of dimension $2m$ which are deformation equivalent to Hilbert schemes of $m$ points on a $K3$ surface.

\begin{theorem}[]\label{pi20} Let $X$ be a smooth projective hyperk\"ahler fourfold of $K3^{[2]}$--type. Assume that $X$ has finite-dimensional motive and let $\{\pi_X^i\}$ be a CK decomposition (which always exists under these hypotheses). Then there exists a further splitting
  \[ \pi_X^2 = \pi_X^{2,0} + \pi_X^{2,1}\ \ \ \hbox{in}\ A^4(X\times X)\ ,\]
  where $\pi_X^{2,0}$ and $\pi_X^{2,1}$ are orthogonal idempotents, and $\pi_X^{2,1}$ is supported on $C\times D\subset X\times X$, where $C$ and $D$ are a curve, resp.\ a divisor on $X$.
  The action on cohomology verifies
  \[ (\pi_X^{2,0})_\ast H^\ast(X) = H^2_{tr}(X)\ ,\]
  where $H^2_{tr}(X)\subset H^2(X)$ is defined as the orthogonal complement of $\NS(X)$ with respect to the Beauville--Bogomolov--Fujiki form. The action on Chow groups verifies
  \[  (\pi_X^{2,0})_\ast A^2(X) = (\pi_X^2)_\ast A^2_{}(X)\ .\]
  \end{theorem}

  \begin{proof} It is known that $X$ verifies the Lefschetz standard conjecture $B(X)$ (see \cite{CM}). Combined with finite-dimensionality, this implies the existence of a CK decomposition by Proposition \ref{jannsen}.
  
  For the ``moreover'' statement, one observes that $X$ verifies conditions (*) and (**) of Vial's \cite{V4}, and so \cite[Theorems 1 and 2]{V4} apply. This gives the existence of refined CK projectors $\pi_X^{i,j}$, which act on cohomology as projectors on gradeds for the ``niveau filtration'' $\wt{N}^\ast$ of loc.\ cit. In particular, $\pi_X^{2,1}$ acts as projector on $\NS(X)$, and $\pi_X^{2,0}$ acts as projector on $H^2_{tr}(X)$. The projector $\pi_X^{2,1}$, being supported on $C\times D$, acts trivially on $A^2(X)$ for dimension reasons: this proves the last equality.
   \end{proof}

\subsection{\llsvs{} eightfolds} \label{construction llsvs}
In this section we briefly recall the construction of the Lehn--Lehn--Sorger--van Straten eightfold (see \cite{LLSS} for additional details). Let $Y \subset \PP^5$ be a smooth cubic fourfold not containing a plane. Twisted cubic curves on $Y$ belong to an irreducible component $M_3(Y)$ of the Hilbert scheme $\textrm{Hilb}^{3m+1}(Y)$; in particular, $M_3(Y)$ is a ten-dimensional smooth projective variety and it is referred to as the Hilbert scheme of \emph{generalized twisted cubics} on $Y$. There exists a hyperk\"ahler eightfold $Z = Z(Y)$ and a morphism $u: M_3(Y) \to Z$ which factorizes as $u = \Phi \circ a$, where $a: M_3(Y) \to Z'$ is a $\PP^2$-bundle and $\Phi: Z' \to Z$ is the blow-up of the image of a Lagrangian embedding $j: Y \hookrightarrow Z$. By \cite{AL} (or alternatively \cite{CLehn}, or \cite[Section 5.4]{KLSV}), the manifold $Z$ is of $K3^{[4]}$-type.

For a point $j(y) \in j(Y) \subset Z$, a curve $C \in u^{-1}(j(y))$ is a singular cubic (with an embedded point in $y$), cut out on $Y$ by a plane tangent to $Y$ in $y$. In particular, $u^{-1}(j(Y)) \subset M_3(Y)$ is the locus of non-Cohen--Macaulay curves on $Y$. If instead we consider an element $C \in M_3(Y)$ such that $u(C) \notin j(Y)$, let $\Gamma_C := \langle C \rangle \cong \PP^3$ be the convex hull of the curve $C$ in $\PP^5$. Then, the fiber $u^{-1}(u(C))$ is one of the $72$ distinct linear systems of aCM twisted cubics on the smooth integral cubic surface $S_C := \Gamma_C \cap Y$ (cf.\ \cite{BK}). This two-dimensional family of curves, whose general element is smooth, is determined by the choice of a linear determinantal representation $[A]$ of the surface $S_C$ (see \cite[Chapter 9]{dolgachev}). This means that $A$ is a $3 \times 3$-matrix with entries in $H^0(\mathcal{O}_{\Gamma_C}(1))$ and such that $\det(A) = 0$ is an equation for $S_C$ in $\Gamma_C$. The orbit $[A]$ is taken with respect to the action of $\left(\textrm{GL}_3(\mathbb{C}) \times \textrm{GL}_3(\mathbb{
C})\right)/\Delta$, where $\Delta = \left\{ (tI_3,tI_3) \mid t \in \mathbb{C}^*\right\}$. Moreover, any curve $C' \in u^{-1}(u(C))$ is such that $I_{C'/S_C}$ is generated by the three minors of a $3 \times 2$-matrix $A_{0}$ of rank two, whose columns are in the $\mathbb{C}$-linear span of the columns of $A$. 

For a cyclic cubic fourfold $Y$ not containing a plane, we now explain the action of the automorphism $g \in \aut(Z(Y))$ of order $3$ introduced in Section \ref{sec: intro}. Let $\sigma \in \aut(Y)$ be the automorphism (\ref{eq: sigma}). The image of the embedding $j: Y \hookrightarrow Z$ is globally invariant under the automorphism: in particular, $g(j(y)) = j(\sigma(y))$. On the other hand, a point $p \in Z \setminus j(Y)$ corresponds to a three-dimensional linear subspace $\Gamma \subset \PP^5$ and a linear determinantal representation $[A]$ of the surface $S = \Gamma \cap Y$. If $A = \left( w_{i,j} \right)$, with $w_{i,j} \in H^0(\mathcal{O}_\Gamma(1))$ for $i,j = 1, 2, 3$, then the image $g(p)$ is the point of $Z \setminus j(Y)$ defined by the datum $\left( \Gamma', [A'] \right)$, where $\Gamma' := \sigma(\Gamma)$ and $A'$ is the $3 \times 3$-matrix whose entries are $w_{i,j} \circ \sigma^{-1} \in H^0(\mathcal{O}_{\Gamma'}(1))$. The fixed locus of $g$ is the Lagrangian submanifold $Z_H \subset Z$ which parametrizes generalized twisted cubics contained in the cubic threefold $Y_H = Y \cap \left\{ x_5 = 0 \right\}$ (see \cite[Proposition 6.7]{CC} and \cite[Proposition 2.9]{ShSol}).

\subsection{Voisin's rational map}
\label{ss:psi}
  
  \begin{proposition}[Voisin {\cite[Proposition 4.8]{V14}}]\label{psi} Let $Y\subset\PP^5$ be a smooth cubic fourfold not containing a plane. Let $F=F(Y)$ be the Fano variety of lines and $Z=Z(Y)$ be the \llsvs{} eightfold of $Y$. There exists a degree $6$ dominant rational map
  \[ \psi\colon\ \ F\times F\ \dashrightarrow\ Z\ .\]

Moreover for suitable symplectic forms $\omega_F, \omega_Z$ on $F$, resp.\ $Z$, we have
  \[ \psi^\ast(\omega_Z)= (pr_1)^\ast(\omega_F) -(pr_2)^\ast(\omega_F)\ \ \ \hbox{in}\ H^{2,0}(F\times F)\ \]
(where $pr_j$ denotes the projection on the $j$-th factor). 
   \end{proposition}

We briefly recall the geometric construction of the rational map $\psi$. Let $(l,l') \in F \times F$ be a generic point, so that the two lines $l, l'$ on $Y$ span a three-dimensional linear space in $\mathbb{P}^5$. For any point $x \in l$, the plane $\langle x, l'\rangle$ intersects the smooth cubic surface $S = \langle l, l' \rangle \cap Y$ along the union of $l'$ and a residual conic $Q'_x$, which passes through $x$. Then, $\psi(l,l') \in Z$ is the point corresponding to the two-dimensional linear system of twisted cubics on $S$ linearly equivalent to the rational curve $l \cup_x Q'_x$ (this linear system actually contains the $\mathbb{P}^1$ of curves $\left\{ l \cup_x Q'_x \mid x \in l\right\}$).

By \cite[Remark 3.8]{LLMS}, the involution $\iota \in \aut(Z)$ introduced in Section \ref{sec: intro} acts on the subset $Z \setminus j(Y)$ as follows. A point $p \in Z \setminus j(Y)$ corresponds to a linear determinantal representation $[A]$ of a smooth cubic surface $S = Y \cap \mathbb{P}^3$, or equivalently to a \emph{six} of lines $(e_1, \ldots, e_6)$ on $S$ (see \cite[Section 9.1.2]{dolgachev} and \cite{BK}). Then, $\iota(p)$ is the point of $Z \setminus j(Y)$ associated with the (unique) six $(e'_1, \ldots, e'_6)$ on the same surface $S$ which forms a \emph{double-six} together with $(e_1, \ldots, e_6)$. In terms of linear determinantal representations, $\iota(p)$ corresponds to the transposed representation $[A^t]$ of the surface $S$ by \cite[Proposition 3.2]{BK}. It can be checked that, if $C$ is a twisted cubic curve on $S$ in the linear system parametrized by the point $p$, then, for any quadric $\mathcal{Q} \subset \langle C \rangle$ containing $C$, the residual intersection $C'$ of $S$ with $\mathcal{Q}$ belongs to the fiber $u^{-1}(\iota(p))$ (see for instance \cite[Footnote 47]{debarre}). By the geometric description of the rational map $\psi$, we conclude 
\begin{equation}\label{rmk: compatibility involution}
	\iota(\psi(l,l')) = \psi(l',l)
\end{equation}

In fact, for $x \in l$ and $x' \in l'$, the reducible quadric $\mathcal{Q} = \langle x, l'\rangle \cup \langle x', l \rangle$ intersects the surface $S$ along the union of $l \cup_x Q'_x$ and $l' \cup_{x'} Q_{x'}$.

   \begin{remark}The indeterminacy locus of the rational map $\psi$ is the codimension $2$ locus $I\subset F\times F$ of intersecting lines (see \cite[Theorem 1.2]{Mura}). Let $Y$ be a cubic fourfold not containing a plane and suppose that $Y$ has finite-dimensional motive.
 Unfortunately, we are {\em not\/} able to prove that the \llsvs{} eightfold $Z=Z(Y)$ also has finite-dimensional motive. It has now been computed by Chen \cite[Theorem 1.1]{Chen} that blowing up the incidence locus $I\subset F\times F$ gives a resolution of indeterminacy of Voisin's map $\psi$ (cf.\ subsection \ref{ss:psi} below), but we still do not know whether $I\subset F\times F$ has finite-dimensional motive.
 
  As we will see in the proof of Proposition \ref{prop:split}, we have at least a weaker result: there exists a submotive $M\subset h(Z)$ such that $M$ is finite-dimensional and $M$ is responsible for the $0$-cycles, i.e. $A^8(M)=A^8(Z)$. Also, in the special case where $Y$ is a Pfaffian cyclic cubic not containing a plane, the eightfold $Z(Y)$ has finite-dimensional motive (see the proof of Proposition \ref{pfaf}).
  
A possible approach to the question of expressing the motive of $Z(Y)$ in terms of the one of $Y$ is to exploit the fact that $Z(Y)$ can be interpreted by \cite{LLMS} as a moduli space of semistable objects, with respect to a suitable Bridgeland stability condition, in the Kuznetsov component of $D^b(Y)$ (cf.\ also \cite{LPZ} and \cite[Section 2.2]{Chen}). 
Provided one can adapt the argument of \cite[Theorem 0.1 and Remark 2.1]{Bul} (replacing the K3 category of loc.\ cit.\ by the Kuznetsov category of a cubic fourfold), this interpretation
would give a nice motivic relation between $Z(Y)$ and powers of $Y$ (in particular, $Z(Y)$ would have finite-dimensional motive if $Y$ has). This approach has been worked out in \cite{FFZ}.
\end{remark}
  
 \section{Decomposition of $A^8(Z)$} \label{ss:vois} 
  
  In this section we construct a \splitting{} of $A^8(Z)$ for \llsvs{} eightfolds associated with cyclic cubic fourfolds not containing a plane; the key ingredient will be the finite-dimensionality proven in \cite{fam}. Later, in Section \ref{subsec-Pfaffians}, we give a similar construction in the special case of Pfaffian cubic fourfolds not containing a plane based on \cite{V6}.
  
 \begin{definition}[Voisin \cite{V14}]
 Let $X$ be a hyperk\"ahler variety of dimension $n=2m$. One defines $S_j (X)\subset X$ as the set of points whose orbit under rational equivalence has dimension $\ge j$.
 The descending filtration $S_\ast$ on $A^n(X)$ is defined by letting $S_j A^n(X)$ be the subgroup generated by points $x\in S_j (X)$.
 \end{definition}
 
 \begin{remark} As explained in \cite{V14}, the expectation is that the filtration $S_\ast$ is opposite to the Bloch--Beilinson filtration (which conjecturally exists for all smooth projective varieties) and thus provides a splitting{} of it. This is motivated by Beauville's speculations in \cite{Beau3} (cf.\ \cite{V12}). 

 \end{remark}
 
  \begin{proposition}\label{prop:split} Let $Y\subset\PP^5$ be a cubic fourfold not containing a plane, let $F=F(Y)$ be the Fano variety of lines and let $Z=Z(Y)$ be the associated \llsvs{} eightfold, such that $F$ has finite-dimensional motive.
  
  \noindent
  (\rom1)(Shen--Vial \cite{SV}) There exist mutually orthogonal idempotents $\Pi_{F\times F}^8, \Pi_{F\times F}^{10},\ldots,\Pi_{F\times F}^{16}\in A^8(F^4)$ with the property that
    \[ A^8_{(j)}(F\times F):= \bigoplus_{k+\ell=j} A^4_{(k)}(F)\otimes A^4_{(\ell)}(F) = (\Pi_{F\times F}^{16-j})_\ast A^8(F\times F)\ .\]
    (Here $A^4_{(\ast)}(F)$ refers to the Fourier decomposition of Theorem \ref{sv}.)
    
    Moreover, the idempotent $\Pi_{F\times F}^{k}$ acts on cohomology as projector on $ H^k_{}(F\times F)$.
    
  \noindent
  (\rom2) There exist idempotents $\Pi_{Z}^8,\Pi_{Z}^{10},\ldots,\Pi_{Z}^{16}\in A^8(Z\times Z)$, such that
    \[  (\Pi_{Z}^8+ \Pi_{Z}^{10}+\cdots+ \Pi_{Z}^{16})_\ast=\ide\colon\ \ \ A^8(Z)\ \to\  A^8(Z)\ \]
    and  
    \[ (\Pi_{Z}^k\circ \Pi_Z^\ell)_\ast A^8(Z)=0\ \ \ \forall k\not=\ell\ .\]

   Moreover, the idempotent ${}^t \Pi_Z^k$ acts as the identity on $H^{16-k}_{tr}(Z)$ (where $H^{16-k}_{tr}(Z)$ denotes the orthogonal complement of the algebraic part of the cohomology $N^1$ under the cup product pairing).
   \end{proposition}

  \begin{proof} 
  (\rom1) Let $\{\Pi_F^i\}$ be the CK decomposition given by Theorem \ref{sv}, such that $A^4_{(k)}(F)=(\Pi_F^{8-k})_*A^4(F)$. We consider the product CK decomposition 
    \begin{equation}\label{proddec} \Pi_{F\times F}^{16-j}:=\displaystyle\sum_{\substack{k+l=j\\ k,l\leq 4}}\Pi_F^{8-k}\times\Pi_F^{8-l}\ .\end{equation}
  It is easy to check that these are mutually orthogonal idempotents in $A^8(F^4)$, lifting the K\"unneth components in cohomology.

\noindent
(\rom2) There is a map of motives
  \[ \bar{\Gamma}_\psi\colon\ \ \ h(F\times F)\ \to\ h(Z)\ \ \ \hbox{in}\ \MM_{\rm rat}\ ,\]
  where $\bar{\Gamma}_\psi\in A^8(F\times F\times Z)$ denotes the closure of the graph of the rational map $\psi$ of Subsection \ref{ss:psi}.
  The projectors of (\rom1) define motives
  \[ h^k(F\times F):= (F\times F, \Pi^k_{F\times F},0)\ \ \ \in\ \MM_{\rm rat}\ \ \ \ \ (k=8,\ldots,16)\ .\]
  
  Let us now consider the map of motives
  \[ \bar{\Gamma}_\psi\colon\ \ \ h^k(F\times F)\ \to\ h(Z)\ \ \ \hbox{in}\ \MM_{\rm rat}\ \ \ \ \ (k=8,\ldots,16)\ .\]  
  
   Using finite-dimensionality of $h^k(F\times F)$ and applying \cite[Lemma 3.6]{V3} to this map, one obtains a \splitting{}
     \begin{equation} \label{eq-motivesplitting} h(Z)= M_k\oplus M_k^\prime\ \ \ \hbox{in}\ \MM_{\rm rat}\ ,\end{equation}                                                                                                        
     such that the map from $h^k(F\times F)$ to $M_k$ (induced by $\bar{\Gamma}_\psi$) admits a right-inverse. This gives a \splitting{}
     \[ h^k(F\times F)=N_k\oplus N_k^\prime\ \ \ \hbox{in}\ \MM_{\rm rat}\ \]
     such that $\bar{\Gamma}_\psi$ induces an isomorphism $N_k\cong M_k$, and so in particular $M_k$ is finite-dimensional. On the level of zero-cycles, the effect of this is that there is a \splitting{}
     \[ A^8_{(16-k)}(F\times F)=A^8(N_k)\oplus A^8(N_k^\prime)\ ,\]
     where $A^8(N_k)=\psi^\ast \psi_\ast A^8_{(16-k)}(F\times F)$ and $A^8(N_k^\prime)=(\ker \psi_\ast)\cap A^8_{(16-k)}(F\times F)$. The isomorphism of motives $N_k\cong M_k$ induces an isomorphism of Chow groups 
     \[ \psi_\ast\colon \ \ A^8(N_k)\ \xrightarrow{\cong}\ A^8(M_k)\ .\]
     Since we know that 
     \[ \psi_\ast\colon\  A^8(F\times F)=\sum_{k=8}^{16} (\Pi^{k}_{F\times F})_\ast A^8(F\times F)\  \to\ A^8(Z)\]
      is surjective, it follows that $\sum_{k=8}^{16} A^8(M_k)=A^8(Z)$ and so $A^8(M_k^\prime)=0$. 
 The projectors defining the motives $M_k$, $k=8,\ldots,16$, give the required $\Pi_Z^k$. 
 
 For any $k\not=\ell$, consider the commutative diagram
   \[ \begin{array}[c]{ccc}       A^8(F\times F) &\xrightarrow{\psi_\ast}& A^8(Z)\\
                                           &&\\
                                          {\scriptstyle  (\Pi^\ell_{F\times F})_\ast}   \downarrow
                                          \ \ \ \ \ \ \ \ \ 
                                            &&\ \  \ \ \ \ \downarrow   {\scriptstyle (\Pi^\ell_Z)_\ast} \\
                                          &&\\
                                        A^8(F\times F) &\xrightarrow{\psi_\ast}& A^8(Z)\\
                                           &&\\
                                           {\scriptstyle  (\Pi^k_{F\times F})_\ast} \downarrow 
                                           \ \ \ \ \ \  \ \ \ && \ \ \ \  \ \ \downarrow   {\scriptstyle (\Pi^k_{Z})_\ast} \\
                                          &&\\
                                        A^8(F\times F) &\xrightarrow{\psi_\ast}& A^8(Z)\\
                                             \end{array}\]
                          Since the $\Pi^k_{F\times F}$ of (\rom1) are mutually orthogonal, it follows that $ \Pi_{Z}^k\circ \Pi_Z^\ell$ acts as $0$ on Chow groups.                
    
 For the action in cohomology, we notice that ${}^t \Pi^k_{F\times F}$ acts as projector on $H^{16-k}(F\times F)$. The claimed behaviour of
 ${}^t \Pi^k_Z$ follows from the commutative diagram
     \[ \begin{array}[c]{ccc}       H^{16-k}(F\times F) &\leftrightarrows& H^{16-k}(Z)\\
                                           &&\\
                                          {\scriptstyle  ({}^t \Pi^k_{F\times F})_\ast}   \downarrow
                                          \ \ \ \ \ \ \ \ \ 
                                            &&\ \  \ \ \ \ \downarrow   {\scriptstyle ({}^t \Pi^k_Z)\ast} \\
                                          &&\\
                                        H^{16-k}(F\times F) &\leftrightarrows& H^{16-k}(Z)\ \\
                                             \end{array}\]
                                   (where horizontal arrows are given by $\psi_\ast$ and $\psi^\ast$), plus the fact that $\psi_\ast \psi^\ast=6 \ide$ on $H^\ast_{tr}(Z)$.         
   \end{proof}

 \begin{corollary}\label{cor:split} Let $Y, F, Z$ be as in Proposition \ref{prop:split}, and set $A^8_{(j)}(Z):=(\Pi_Z^{16-j})_\ast A^8(Z)$. This induces a \splitting{}
   \[ A^8(Z)=A^8_{(0)}(Z)\oplus A^8_{(2)}(Z)\oplus \cdots A^8_{(8)}(Z)\ ,\] 
   such that
    \[ \psi_\ast A^8_{(2i)}(F\times F)= A^8_{(2i)}(Z)\ .\]
    The piece $A^8_{(2i)}(Z)$ is isomorphic to $S_{4-i}A^8(Z)/S_{5-i}A^8(Z)$.
   \end{corollary} 
   
   \begin{proof}
   The first part of the statement follows immediately from the definition of $A^8_{(2j)}(F\times F)$ and of the idempotents $\Pi_Z^8,\ldots,\Pi_Z^{16}$ of Proposition \ref{prop:split}.
   
In \cite[Corollary 4.9]{V14} Voisin shows that $S_kA^8(Z)\cong \psi_*\left(\sum_{i+j=k}S_iA^4(F)\otimes S_jA^4(F)\right)$ 
. Moreover, \cite[Proposition 4.5]{V14} says that $S_1A^4(F)=A^4_{(0)}(F)\oplus A^4_{(2)}(F)$ and that $S_2A^4(F)=A^4_{(0)}(F)$. The combination of these two facts gives the desired isomorphism, as we show in the case $i=1$.
Indeed, $$S_4A^8(Z)\cong \psi_*\left(\sum_{i+j=4}S_iA^4(F)\otimes S_jA^4(F)\right)=\psi_*\left(S_2A^4(F)\otimes S_2A^4(F)\right)=\psi_*(A^4_{(0)}(F)\otimes A^4_{(0)}(F))$$ $$=\psi_*(A^8_{(0)}(F\times F)),$$ $$S_3A^8(Z)\cong \psi_*\left(\sum_{i+j=3}S_iA^4(F)\otimes S_jA^4(F)\right)=\psi_*\left(S_2A^4(F)\otimes S_1A^4(F)+S_1A^4(F)\otimes S_2A^4(F)\right)$$ $$=\psi_*(A^4_{(0)}(F)\otimes A^4_{(0)}(F)\oplus A^4_{(0)}(F)\otimes A^4_{(2)}(F)\oplus A^4_{(2)}(F)\otimes A^4_{(0)}(F) )=\psi_*(A^8_{(0)}(F\times F)\oplus A^8_{(2)}(F\times F)),$$ so that $S_3A^8(Z)/S_4A^8(Z)\cong \psi_*(A^8_{(2)}(F\times F))\cong A^8_{(2)}(Z)$.  
   \end{proof}

 By Theorem \ref{findim},  Proposition \ref{prop:split} and Corollary \ref{cor:split} hold in particular when $Y$ is a cyclic cubic fourfold not containing a plane.
  We now give an interpretation of the piece $A^8_{(0)}(Z)$ 
   of our decomposition:
   
   \begin{proposition}\label{piece0} Let $Y\subset\PP^5$ be a cyclic cubic not containing a plane and let $Z=Z(Y)$ be the \llsvs{} eightfold.  
   One has
     \[ A^8_{(0)}(Z)=\QQ[c_8(T_Z)]=\QQ[h^8]=\QQ[Y^2]=\QQ[ (Z_H)^2]\ ,\]
     where $h\in A^1(Z)$ is a hyperplane section, $Y\in A^4(Z)$ is (the class of) the cubic $j(Y) \subset Z$ and $Z_H\in A^4(Z)$ is (the class of) the variety of twisted cubics contained in a hyperplane section.
      \end{proposition}
            
     \begin{proof} This uses some results concerning the {\em Franchetta property\/} \cite{FLV}. By construction, we have
       \[ A^8_{(0)}(Z)=\psi_\ast A^8_{(0)}(F\times F)\ .\]
       Also, we know that $\psi_\ast\psi^\ast=6\ide\colon A^8(Z)\to A^8(Z)$.
       Thus, to show that some $0$-cycle $a\in A^8(Z)$ lies in the piece $A^8_{(0)}(Z)$, it suffices to prove that 
         \[ \psi^\ast(a)\ \ \in\ A^8_{(0)}(F\times F)  \ .\]
  
      To conclude, we now use the $i=8$ case of the following
     
     \begin{claim} Let $\FF\to B$ denote the universal family of Fano varieties of lines on cubic fourfolds, and let $\Gamma\in A^i(\FF\times_B \FF)$. Then, for any $b\in B$ the restriction
     $\Gamma\vert_b\in A^i(F_b\times F_b)$ lies in $A^i_{(0)}(F_b\times F_b)$.
        \end{claim}    
    
    Since the classes $c_8(T_Z), h, Y, Z_H\in A^\ast(Z)$ are universally defined (i.e.\ they exist as relative cycles in $A^\ast(\ZZZ)$, where $\ZZZ\to B^\circ$ is the universal family of \llsvs{} eightfolds, and $B^\circ\subset B$ is the open parametrizing cubics not containing a plane), and Voisin's rational map can be defined family-wise (cf.\ \cite[Proof of Theorem 1.11]{FLV}), the claim settles Proposition \ref{piece0}.
    
    It remains to prove the claim. Let $$\Pi^k_{\FF\times_B\FF}\in A^{16}(\FF\times_B\FF\times_B\FF\times_B\FF)$$ be relative correspondences that restrict to the product CK decomposition $\Pi^k_{F\times F}$ of (\ref{proddec}) (to see that such relative correspondences exist, cf.\ \cite[Section A.2]{FLV}). Given $\Gamma\in    
    A^i(\FF\times_B \FF)$, let us look at
      \[ (\Pi^k_{\FF\times_B\FF})_\ast(\Gamma)\ \ \in\ A^i(\FF\times_B \FF) \ \ \ (k\not=2i)\ .\]
      (For the formalism of relative correspondences, cf.\ \cite[Chapter 8]{MNP}).
      This is fibrewise homologically trivial, and so \cite[Theorem 1.10]{FLV} implies that this is fibrewise rationally trivial:
        \[    (\Pi_{F\times F}^k)_\ast(\Gamma\vert_b)=\left.\Bigl( (\Pi^k_{\FF\times_B\FF})_\ast(\Gamma)\Bigr){}\right|_b=0\ \ \hbox{in}\ A^i(F_b\times F_b) \ \ \ \forall k\not=2i\ .\]
        This means that
        \[ \Gamma\vert_b\ \ \in\ A^i_{(0)}(F_b\times F_b)\ ,\]
        proving the claim.            
     \end{proof}

   The piece $A^8_{(2)}(Z)$ of the decomposition also admits an easy interpretation:   
   
  \begin{proposition}\label{piece2} Let $Y\subset\PP^5$ be a cyclic cubic not containing a plane and let $Z=Z(Y)$ be the \llsvs{} eightfold. There is a (correspondence-induced) isomorphism
    \[ A^8_{(2)}(Z)\cong A^3_{hom}(Y)\ .\]
   \end{proposition}
   
   \begin{proof} Let $F=F(Y)$ be the Fano variety of lines. As is well--known (cf.\ \cite{BD}), the universal line $L\in A^3(Y\times F)$ induces the isomorphism
     \[ L_\ast\colon\ \ H^4_{tr}(Y)\ \xrightarrow{\cong}\ H^2_{tr}(F)(1)\ .\]
     Proposition \ref{psi} gives a (correspondence-induced) isomorphism
     \[   \psi_\ast (pr_1^\ast -pr_2^\ast)    \colon\ \ H^2_{tr}(F)\ \xrightarrow{\cong}\ H^2_{tr}(Z)\ .\]
     Using hard Lefschetz, we know that there is a (correspondence-induced) isomorphism
      \[ H^2_{tr}(Z)\ \xrightarrow{\cong}\ H^{14}_{tr}(Z)(-6)\ .\]
     
   Composing these isomorphisms (and invoking Manin's identity principle \cite[Section 2.3]{Sc}), we obtain an isomorphism of homological motives
    \[  h^4(Y)\oplus \bigoplus \LLL(\ast)\ \xrightarrow{\cong}\ h^{14}(Z)(-5)\oplus \bigoplus \LLL(\ast)\ \ \ \hbox{in}\ \MM_{\rm hom}\ \]
    (where $\LLL$ denotes the Lefschetz motive). 
    
   Now let us consider the Chow motive $(Z,\Pi_Z^{14},0)$ constructed above. Since from the \splitting{} (\ref{eq-motivesplitting}) we get $H^{14}_{tr}(M^\prime)=0$, we have an isomorphism of homological motives
   $$ (Z,\Pi_Z^{14},0)\oplus \bigoplus \LLL(\ast)\cong h^{14}(Z).$$ Using finite-dimensionality of $h^4(Y)$ and of $(Z,\Pi_Z^{14},0)$, we thus obtain an isomorphism of Chow motives
   \[ h^4(Y)\oplus \bigoplus \LLL(\ast)\ \xrightarrow{\cong}\ (Z,\Pi_Z^{14},-5)\oplus \bigoplus \LLL(\ast)\ \ \ \hbox{in}\ \MM_{\rm rat}\ .\]
   Taking $A^3_{hom}()$ on both sides, this proves the proposition.
      \end{proof} 
  
  For technical reasons, later on we will need to work with some subprojectors of the $\left\{\Pi_{Z}^i\right\}$, which we construct by using the refined CK decomposition introduced in Section \ref{pi20}.

\begin{proposition}\label{prop-refinedCK}
Let $Y\subset\PP^5$ be a cyclic cubic fourfold not containing a plane, let $F=F(Y)$ be the Fano variety of lines and let $Z=Z(Y)$ be the associated \llsvs{} eightfold.
  
  \noindent
  (\rom1) There exist mutually orthogonal idempotents $\lbrace P_{F\times F}^i\rbrace\in A^8(F^4)$ which are subprojectors of the $\{\Pi_{F\times F}^i\}$ of Proposition \ref{prop:split}, such that
    \[ A^8_{(j)}(F\times F) = (P_{F\times F}^{16-j})_\ast A^8(F\times F)\ .\]   
    
 Moreover, the idempotent ${}^t P_{F\times F}^{16-k}$ acts on cohomology as projector on $\sym^{k/2} H^2_{tr}(F\times F)$.    
   
   (\rom2) There exist mutually orthogonal idempotents $\lbrace P_{Z}^i\rbrace\in A^8(Z\times Z)$ which are subprojectors of the $\{\Pi_{Z}^i\}$ of Proposition \ref{prop:split}, such that
    \[ A^8_{(j)}(Z) = (P_{Z}^{16-j})_\ast A^8(Z)\ .\]     
    
    Moreover, the idempotent ${}^t P_Z^{14}$ acts as projector on $H^2_{tr}(Z)$.
\end{proposition}

\begin{proof}

\noindent
(\rom1)
We now construct subprojectors of the $\{\Pi_{F\times F}^i\}$ . We start from the refined CK decomposition of Theorem \ref{pi20} $\Pi^2_F=\Pi^{2,0}_F+\Pi^{2,1}_F$ in $A^4(F^2)$ and we consider $\Pi_F^{4,0,0}:= \Delta^{sm}_F\circ (\Pi^{2,0}_F\times \Pi^{2,0}_F)\circ \Psi$, where $\Psi$ is the inverse of $\Delta^{sm}_F:h^2(F)\otimes h^2(F)\rightarrow h^4(F)$ in $\mathcal{M}_{\rm hom}$ (by finite dimensionality we get that it is also an inverse in $\MM_{\rm rat}$).  Recall that $(\Pi_F^{4,0,0})_*(A^4(F))=(\Pi_F^4)_*(A^4(F))=A^4_{(4)}(F)$.

Define now:
\begin{equation*}
\begin{split}
P_{F\times F}^8=&\Pi_F^{4,0,0}\times \Pi_F^{4,0,0},\\ P_{F\times F}^{10}=&{}^t(\Pi^{4,0,0}_F\times \Pi^{2,0}_F+ \Pi^{2,0}_F\times \Pi^{4,0,0}_F ),\\ P_{F\times F}^{12}=&{}^t(\Pi^{4,0,0}_F\times \Pi^{0}_F+ \Pi^{0}_F\times \Pi^{4,0,0}_F +\Pi^{2,0}_F\times \Pi^{2,0}_F),\\ P_{F\times F}^{14}=&{}^t(\Pi^0_F\times \Pi^{2,0}_F+ \Pi^{2,0}_F\times \Pi_F^0),\\ P_{F\times F}^{16}=&\Pi^8_F\times \Pi^8_F={}^t(\Pi^0_F\times \Pi_F^0).
\end{split}
\end{equation*}

To check that the action of $P_{F\times F}^j$ on 0-cycles coincides with the action of $\Pi_{F\times F}^j$, one reduces to the same statement for
$\{\Pi_{F}^j\}_{j=2,4}$ versus $\{ \Pi_F^{2,0},\Pi_F^{4,0,0}\}$. This last statement is proven in \cite[Proof of Proposition 3.8]{nonsymp3}.

Idempotency follows from the idempotency and orthogonality of $\Pi_F^0$, $\Pi_F^{2,0}$ and $\Pi_F^{4,0,0}$. Clearly $P_{F\times F}^{16-j}$ is a subprojector of $\Pi_{F\times F}^{16-j}$, so they are also mutually orthogonal.

Finally, by definition $(\Pi_F^{2,0})_*H^*(F)=H^2_{tr}(F)$, hence also $(\Pi_F^{4,0,0})_*H^*(F)=\sym^2 H^2_{tr}(F)$. As a consequence, we obtain that $({}^t P_{F\times F}^{16-k})_*H^*(F\times F)=\sym^{k/2} H^2_{tr}(F\times F)$.

\noindent
(\rom2)
The $P_Z^i$ are constructed by applying the ``splitting construction'' of Proposition \ref{prop:split}(\rom2) to the $P_{F\times F}^i$.
Since ${}^t P_{F\times F}^{14}$ acts as projector on $H^2_{tr}(F\times F)$, and $\psi_\ast\psi^\ast=6\ide$ on $H^2_{tr}(Z)$, it follows
 that $P_{Z}^{14}$ acts as projector on $\psi_\ast H^2_{tr}(F\times F)=H^2_{tr}(Z)$.
\end{proof}

 \subsection{Pfaffian cubics}\label{subsec-Pfaffians}
 
 We consider now the family of Pfaffian cubic fourfolds. For the general element $Y$ in this family, Beauville and Donagi \cite{BD} have shown that there is an isomorphism between $F(Y)$ and the Hilbert scheme $S^{[2]}$ of the associated K3 surface $S$ of degree 14.
 
 \begin{proposition}\label{K3iso} Let $Y\subset\PP^5$ be a Pfaffian cubic fourfold not containing a plane. Let $S$ be the associated K3 surface of degree $14$ and let $Z=Z(Y)$ be the \llsvs{} eightfold.
 There is an isomorphism of motives 
   \[ h(Z)\cong h(S^{[4]})\ \ \ \hbox{in}\ \MM_{\rm rat}\ .\]
   This induces an isomorphism of Chow rings $A^\ast(Z)\cong A^\ast(S^{[4]})$.
  \end{proposition}
  
  \begin{proof} For a {\em general\/} Pfaffian cubic $Y$, it is known (cf.\ \cite{AL}) that there is a birational map $Z\dashrightarrow S^{[4]}$ (precisely, the condition on $Y$ is that $Y$ does not contain a plane and $S$ does not contain a line). Thus, for a {\em general\/} Pfaffian cubic the isomorphism of motives and of Chow rings follows from Rie\ss's result \cite{Rie}. A standard spreading argument (cf.\ \cite[Lemma 3.2]{Vo}) then shows that these isomorphisms extend to {\em all\/} Pfaffian cubics.
  \end{proof}

 For cyclic cubics that are Pfaffian, which exist as shown in \cite[Proposition 5.1]{BCS3}, one can improve on Proposition \ref{prop:split}(\rom2):
 
  \begin{proposition}\label{pfaf} Let $Y\subset\PP^5$ be a Pfaffian cyclic cubic fourfold not containing a plane and let $Z=Z(Y)$ be the \llsvs{} eightfold. There exist mutually orthogonal idempotent correspondences $R_Z^8,R_Z^{10},\ldots,R_Z^{16}\in A^8(Z\times Z)$, such that
    \[  (R_Z^{16-j})_\ast A^8(Z)= A^8_{(j)}(Z)\ .\]
    
    Moreover, the idempotent ${}^t R_Z^{16-k}$ acts on cohomology as projector on $\sym^{k/2} H^2_{tr}(Z)$.   
  \end{proposition}
  
  \begin{proof} A first observation  is that the Fano variety $F(Y)$ is isomorphic to a Hilbert square $S^{[2]}$, where $S$ is a K3 surface \cite{BD}. Also, we know that $F(Y)\cong S^{[2]}$ has finite-dimensional motive (Theorem \ref{findim}). This implies that $S$ also has finite-dimensional motive and thus the eightfold $Z\cong S^{[4]}$ has finite-dimensional motive.
  
 Theorem \ref{K3m} now ensures that $Z$ has a self-dual MCK decomposition $\{\pi^j_Z\}$. This induces the same \splitting{} on $0$-cycles:
 
     \begin{lemma} Let $Z$ be as in Proposition \ref{pfaf}.
 The \splitting{} of $A^8(Z)$ induced by the MCK decomposition of Theorem \ref{K3m} coincides with the \splitting{} of Corollary \ref{cor:split}:
  \[ A^8_{(j)}(Z):=(\Pi_Z^{16-j})_\ast A^8(Z)= (\pi^{16-j}_Z)_\ast A^8(Z)\ .\]
 \end{lemma}
 
 \begin{proof}  
For any Hilbert scheme $S^{[m]}$ of a K3 surface $S$, Vial's \splitting{} of $A_0(S^{[m]})$ as in Theorem \ref{K3m} coincides with Voisin's filtration $S_\ast$ (this is observed in \cite[End of Section 4.1]{V14}).
\end{proof}
 
 Next, let us consider the decomposition $\pi^2_Z=\pi^{2,0}_Z+\pi^{2,1}_Z$ of Theorem \ref{pi20}. Using the argument of \cite[Proof of Proposition 3.8]{nonsymp3}, this decomposition induces
 idempotents $T_Z^k\in A^8(Z\times Z)$, $k=0,\ldots,8$, acting on cohomology as projectors on $\sym^{k/2} H^2_{tr}(Z)$. We define $R_Z^k$ to be the transpose
 $R_Z^k:={}^t T_Z^{16-k}$. The support condition for $\pi_Z^{2,1}$ in Theorem \ref{pi20} implies that $(R_Z^{14})_\ast  =(\pi_Z^{14})_\ast\colon A^8(Z)\to A^8(Z)$.
 For $k=8,\ldots, 12$, we use the fact that $A^8(Z)=A^8(S^{[4]})=A^8(S^4)^{\Sy_4}$ (cf.\ \cite{dCM}) to lift the computations to $S^4$, since by construction of Vial \cite{V6} we have a commutative diagram
 \[
 \xymatrix{
 A^8(S^4)^{\Sy_4}\ar@{->}[r]^{(\pi_{S^4}^k)_*}&A^8(S^4)^{\Sy_4}\\
 A^8(Z)\ar@{->}[u]^{\simeq}\ar@{->}[r]^{(\pi_{Z}^k)_*}&A^8(Z)\ar@{->}[u]^{\simeq}
 }
 \]
 and a similar diagram holds for $R_{S^4}^k$ and $R_{Z}^k$ by \cite{nonsymp3}. On $S^4$ it is readily checked that $R_{S^4}^k$ and $\pi_{S^4}^k$ differ only by terms involving $\pi^{2,1}_S$ which is not acting on zero-cycles because of its support property, thus we conclude again that $(R_{S^4}^k)_\ast  =(\pi_{S^4}^k)_\ast\colon A^8({S^4})\to\ A^8({S^4})$
   and so the $R_Z^k$'s do the job.  
  \end{proof}
  
  \begin{remark}
   It remains an open question whether the projectors $R_Z^k$ coincide with the projectors $P_Z^k$ of Proposition \ref{prop-refinedCK}, since it is not clear whether $P_Z^k$ acts on cohomology as projector on $\sym^{k/2} H^2_{tr}(Z)$.   
  \end{remark}

  \begin{remark}
   All the results that we have obtained for Pfaffian cubic fourfolds go through for all other special cubic fourfolds $Y$, in the sense of Hassett \cite{Hassett}, for which the Fano variety of lines $F(Y)$ is isomorphic to $S^{[2]}$, with $S$ the associated K3 surface.
  \end{remark}

  \section{Proofs of main results}\label{section - Main results}
  
 We now prove the main results stated in Section \ref{sec: intro}.

  \begin{proof}[Proof of Theorem \ref{main}] Let $F=F(Y)$ denote the Fano variety of lines on $Y$. We have seen that there is a commutative diagram
    \[ \begin{array}   [c]{ccc} F\times F & \dashrightarrow & Z\\
    &&\\
         \ \ \ \ \  \downarrow{\scriptstyle (g,g)} && \downarrow{\scriptstyle g}\\
         &&\\
          F\times F & \dashrightarrow & Z\\  
          \end{array}\]
          (where the horizontal arrows are Voisin's rational map $\psi$).

  Let us use the shorthand notation
    \[ \Delta_G:= {1\over 3} \sum_{i=1}^3 \Gamma_{g^i}\ \ \ \in\ A^8\bigl((F\times F)\times (F\times F)\bigr)\ ,\] 
    where $G\cong\ZZ_3$ is the group generated by the diagonal action of $g$ on $F\times F$.

    Since $A^8_{(j)}(Z)= \psi_\ast A^8_{(j)}(F\times F)$, the theorem follows once we prove that
    \[ (\Delta_G)_\ast=\ide\colon\ \ A^8_{(j)}(F\times F)\ \to\ A^8(F\times F) \ \ \ \hbox{for}\ j=0\ ,\]     
    and
      \[ (\Delta_G)_\ast=0\colon\ \ A^8_{(j)}(F\times F)\ \to\ A^8(F\times F) \ \ \ \hbox{for}\ j\in\{2,4\}\ .\]     
      In view of Proposition \ref{prop-refinedCK}, this is equivalent to proving that
      \begin{equation}\label{goal}  
         (\Delta_G\circ P_{F\times F}^{16-j})_\ast = \begin{cases}    (P_{F\times F}^{16-j})_\ast\colon\ \ \ A^8(F\times F)\ \to\ A^8(F\times F) & \ \hbox{for}\ j=0\ ,\\
                                                               0\colon\ \ \     A^8(F\times F)\ \to\ A^8(F\times F) & \ \hbox{for}\ j\in\{2,4\}\ .\\
                                                               \end{cases}
                                                               \end{equation}
                                                               
         We first establish a lemma:
         
         \begin{lemma}\label{idem} For any $j$, the correspondence $\Delta_G\circ P^{16-j}$ acts idempotently:
         \[ \bigl((\Delta_G\circ P_{F\times F}^{16-j})^{\circ 2}\bigr){}_\ast=(\Delta_G\circ P_{F\times F}^{16-j})_\ast\colon\ \ \ A^8(F^2)\ \to\ A^8(F^2)\ .\]
         \end{lemma}
         
         \begin{proof} The action of $g$ on $A^4(F)$ respects the Fourier decomposition $A^4_{(j)}(F)$, i.e.
         \[ g_\ast A^4_{(j)}(F)\ \ \subset\ A^4_{(j)}(F)\ \ \ \forall j \]
         by \cite[Proposition 3.2]{nonsymp3}. This implies that
         \[   (P_{F\times F}^{16-j}\circ\Delta_G\circ P_{F\times F}^{16-j})_\ast=(\Delta_G\circ P_{F\times F}^{16-j})_\ast\colon\ \ \ A^8(F^2)\ \to\ A^8(F^2)\ .\]   
         Since $\Delta_G$ is obviously idempotent, this proves the lemma.
               \end{proof}

    Let us start by proving the desired result (\ref{goal}) for $j=0$. Since the projector $P_{F\times F}^{16}$ is just the CK projector $\Pi^{8}_{F}\otimes \Pi^8_F$ of Theorem \ref{sv}, and $g$ acts as the identity on $H^8(F)$, we know that
    \[ \Delta_G\circ P_{F\times F}^{16}=\Delta_G\circ ( \Pi^{8}_{F}\otimes \Pi^8_F)=\Pi^{8}_{F}\otimes \Pi^8_F= P_{F\times F}^{16}\ \ \ \hbox{in}\ H^{16}(F^4)\ .\]
    That is, we have
    \[ \Delta_G\circ P_{F\times F}^{16} - P_{F\times F}^{16}\ \ \ \in\ A^8_{hom}(F^4)\ .\]
    Because $F\times F$ has finite-dimensional motive, there exists $N\in\NN$ such that
    \[ \bigl( \Delta_G\circ P_{F\times F}^{16} - P_{F\times F}^{16}\bigr)^{\circ N} =0\ \ \ \in\ A^8_{}(F^4)\ .\]
    But $P_{F\times F}^{16}$ is idempotent (Theorem \ref{sv}), and $\Delta_G\circ P_{F\times F}^{16}$ acts idempotently on $0$-cycles (Lemma \ref{idem}), hence (taking $N$ odd) we find that
    \[ (\Delta_G\circ P_{F\times F}^{16})_\ast = (P_{F\times F}^{16})_\ast\colon\ \ \ A^4_{}(F^2)\ \to\ A^4(F^2)\ ,\]
    proving (\ref{goal}) for $j=0$.  
    
     Next, let us prove the desired result (\ref{goal}) for $j=2$. We consider the correspondence $ P_{F\times F}^{14}\circ \Delta_G$, and we make the following claim:
    
       \begin{claim}\label{claim2} There is equality
      \[ \Delta_G\circ {}^tP_{F\times F}^{14}=0\ \ \ \hbox{in}\ H^{16}(F^4)\ .\]
      \end{claim}  
      
       The claim implies that $ \Delta_G\circ {}^tP_{F\times F}^{14}  \in A^8_{hom}(F^4)$, and so its transpose $P_{F\times F}^{14}\circ \Delta_G $ is also in $A^8_{hom}(F^4)$. In particular, the composition 
       \[ P_{F\times F}^{14}\circ \Delta_G \circ P_{F\times F}^{14}\]
       is in $A^8_{hom}(F^4)$.
               Using finite-dimensionality, this means that there exists $N\in\NN$ such that
       \[ \bigl(   P_{F\times F}^{14}\circ \Delta_G\circ P_{F\times F}^{14} \bigr)^{\circ N}=0\ \ \ \hbox{in}\ A^8(F^4)\ .\]
       Looking at the action on $0$-cycles, and using the fact that 
       \[ ( P_{F\times F}^{14}\circ \Delta_G\circ P_{F\times F}^{14})_\ast     =     ( \Delta_G\circ P_{F\times F}^{14})_\ast\colon\ \ \ A^8(F^2)\ \to\ A^8(F^2) \]
       (proof of Lemma \ref{idem}), this implies that
       \[          ( \Delta_G\circ P_{F\times F}^{14})_\ast= 0\colon\ \ \ A^8(F^2)\ \to\ A^8(F^2) \ ,\]    
       therefore we have established the desired result (\ref{goal}) for $j=2$.
       
       It remains to prove Claim \ref{claim2}. Since ${}^tP_{F\times F}^{14}$ acts on cohomology as projector on $ H^2_{tr}(F\times F)$, 
       it will suffice to prove that
       \[ (\Delta_G)_\ast (a)=0\ \ \ \hbox{in}\ H^2(F\times F)\ \ \ 
        \ \forall a  \in H^2_{tr}(F\times F)\ .\]
        To this end, we consider the eigenspace decomposition of complex-valued cohomology, and define a subspace
        \[ H:=\bigl\{ a\in H^2_{}(F\times F)\ \vert\  a_\C\in  H^2(F\times F,\C)^\nu \oplus H^2(F\times F,\C)^{\nu^2}\bigr\}\ .\]
        The subspace $H$, together with its complexification, is a Hodge structure containing $p_1^\ast(\omega_F)$ and $p_2^\ast(\omega_F)$. But $H^2_{tr}(F\times F)$ is the smallest Hodge substructure containing these two classes, and so $H^2_{tr}(F\times F)\subset H$. It follows that
        \begin{equation}\label{onH2} a_\C\in H^2(F\times F,\C)^\nu \oplus H^2(F\times F,\C)^{\nu^2}\ \ \forall a\in H^2_{tr}(F\times F)\ ,\end{equation}  
        and Claim \ref{claim2} is proven.  
    
    Let us now do the case $j=4$. We make the following claim:
    
    \begin{claim}\label{claim3} There exists $\Gamma\in A^8(F^4)$ such that
      \[ \Delta_G\circ ( \Delta_{F\times F} -\Gamma)\circ {}^tP_{F\times F}^{12} =0\ \ \ \hbox{in}\ H^{16}(F^4)\ ,\]
      and $\Gamma$ factors as $\Gamma=\Upsilon^{\prime\prime}\circ \Upsilon^\prime $, where $\Upsilon^\prime\in A^7(F^2\times T)$, $\Upsilon^{\prime\prime}\in A^{3}(T\times F^2)$ and $T$ is a (not necessarily connected) surface.
      \end{claim}    
      
     This is sufficient: using the finite-dimensionality of $F^2$, there exists $N\in\NN$ such that
     \[  \Bigl(    \Delta_G\circ ( \Delta_{F\times F} -\Gamma)\circ {}^tP_{F\times F}^{12}        \Bigr)^{\circ N}=0\ \ \ \hbox{in}\ A^{8}(F^4)\ .\]    
     Developing (and using Lemma \ref{idem}), we find that
     \[ \Bigl(  \Delta_G\circ {}^tP_{F\times F}^{12} + \Gamma^\prime\Bigr){}_\ast =0\colon\ \ A^8(F^2)\ \to\ A^8(F^2)\ ,\]
     where $\Gamma^\prime$ is a correspondence that factors over $T$. But $\Gamma^\prime$ does not act on $0$-cycles for dimension reasons, and so
     \[ ( \Delta_G\circ {}^tP_{F\times F}^{12})_\ast =0\ \colon\     A^8(F^2)\ \to\ A^8(F^2)\ .\]
     
     It remains to prove Claim \ref{claim3}. First of all we remark that $H^4(F\times F)^+=H^4(F)^+\otimes H^0(F)\oplus (H^2(F)\otimes H^2(F))^+\oplus H^0(F)\otimes H^4(F)^+$, and by Abel--Jacobi isomorphism we also have $H^2(F)_{}\otimes H^2(F)_{}\simeq H^4(Y)_{}^{\otimes 2}$. Moreover, since $F$ is a fourfold of $K3^{[2]}$-type, we have $H^2(F)^{\otimes 2} = H^4(F)$, thus finally we reduce the study of $H^4(F\times F)^+$ to the study of $H^4(F)^+\simeq (H^4(Y)^{\otimes 2})^+$.
     
 We now apply Shioda--Katsura's construction \cite[Proposition 2.4 and Remark 1.10]{SK}. This gives a rational map from $Y\times Y$ to a cubic eightfold $Y_8$, defined by the equation
 $f(x_0,\ldots,x_4) + f(x_5,\ldots,x_9)=0$, and an injection
   \[ \phi= (\phi_{1},\phi_2)\colon\ \   H^8(Y\times Y)^+ \ \hookrightarrow\ H^{8}(Y_8)\oplus W\ ,\]
   where $W\subset H^{2\ast}(Y_H)^{\oplus 2}$ and $Y_H\subset Y$ is the cubic threefold defined by $f(x_0,\ldots,x_4)=0$. 
   Since the construction of loc. cit. is geometric in nature, $\phi_{}$ and its left-inverse are induced by correspondences $(\Phi_1,\Phi_2)$ resp. $(\Omega_1,\Omega_2)$. It is known that $A^j(Y_8)=\QQ$ for $j\ge 6$ \cite{Otw}, which means that there exists a surface $\Sigma$ such that $\Delta_{Y_8}=\Psi_1\circ \Theta_1 $, where $\Psi_1$ (resp. $\Theta_1$) is a correspondence from $Y_8$ to $\Sigma$ (resp. from $\Sigma$ to $Y_8$). Analogously, since $H^{2\ast}(Y_H)$ is algebraic, we obtain a correspondence $\Psi_2$ (resp. $\Theta_2$) from the copies of $Y_H$ giving rise to $W$ to a finite number of points (resp. vice versa). We can now define 
   \[\widetilde{\Upsilon}^{\prime}=\Psi_1\circ \Phi_1+\Psi_2\circ \Phi_2\in A^\ast(Y^2\times \widetilde{T}), \quad \widetilde{\Upsilon}^{\prime\prime}=\Omega_1\circ \Theta_1+\Omega_2\circ \Theta_2\in A^{\ast}(\widetilde{T}\times Y^2)\]
   where $\widetilde{T}$ is the union of several copies of $\Sigma$ and of copies of $\mathbb{P}^2$ representing the points above.
   Finally, composing with the Abel--Jacobi correspondence we obtain correspondences $\Upsilon^\prime\in A^7(F^2\times T)$, $\Upsilon^{\prime\prime}\in A^{3}(T\times F^2)$. Defining
   $\Gamma:=\Upsilon^{\prime\prime}\circ \Upsilon^\prime $, we have that
   \[      \begin{split} \Bigl(\Delta_G\circ ( \Delta_{F\times F} -\Gamma)\circ {}^tP_{F\times F}^{12}    \Bigr){}_\ast H^\ast(F\times F) &=    \Bigl(\Delta_G\circ ( \Delta_{F\times F} -\Gamma)\Bigr){}_\ast H^2(F\times F)^{\otimes 2}\\
                   &=   (\Delta_G){}_\ast  ( H^\perp )\ ,\\
                       \end{split}\]
                    where $H^\perp\subset  H^2(F\times F)^{\otimes 2}$ is the complement to $( H^2(F\times F)^{\otimes 2})^+$. Since the complexification of $H^\perp$ is the direct sum of the eigenspaces 
                    $ (H^2(F\times F,\C)^{\otimes 2})^\nu$ and     $ (H^2(F\times F,\C)^{\otimes 2})^{\nu^2}$, we see that $\Delta_G$ acts as zero, as claimed. \qedhere                                                                      
  \end{proof}

  \begin{proof}[Proof of Theorem \ref{main2}]
  The first thing to prove is that the involution $\iota$ respects the \splitting{} $A^8(Z)=\oplus A^8_{(j)}(Z)$ (this is the analogue of Lemma \ref{idem}).
  
  \begin{lemma}\label{OK} Let $Y\subset\PP^5$ be a cyclic cubic not containing a plane and let $Z=Z(Y)$ be the associated \llsvs{} eightfold.  Let $\iota\in\aut(Z)$ be the involution of Remark \ref{rmk: compatibility involution}. Then
    \[ \iota^\ast A^8_{(j)}(Z)\ \ \subset\ A^8_{(j)}(Z)\ \ \ \forall j\ .\]
    \end{lemma}
    
    \begin{proof} We have seen (Section \ref{ss:psi}, Equation \eqref{rmk: compatibility involution}) that there is a commutative diagram
      \[ \begin{array}[c]{ccc} 
            F\times F & \dashrightarrow & Z\\
            &&\\
         \ \ \ \ \  \downarrow{\scriptstyle \iota_F} && \downarrow{\scriptstyle \iota}\\
         &&\\
          F\times F & \dashrightarrow & Z\\  
          \end{array}\]
          where the horizontal arrows are Voisin's rational map $\psi$ and $\iota_F\colon F\times F\to F\times F$ is the map exchanging the two factors.    
          
          We have also seen that
          \[ A^8_{(j)}(Z)=\psi_\ast \left( \bigoplus_{k+\ell=j} A^4_{(k)}(F)\otimes A^4_{(\ell)}(F)\right) \ .\]
          Since $\iota_F$ respects the \splitting{} on the right-hand side, it follows that $\iota$ respects the \splitting{} on the left-hand side. 
     \end{proof}
   
   The proof of Theorem \ref{main2} is now similar to the one of Theorem \ref{main}. We introduce correspondences
    \[    \Gamma_j:= {1\over 2}\, \bigl( \Gamma_{\iota}-(-1)^{j/2}\Delta_{Z}\bigr)\circ R_Z^{16-j}\ \ \ \in A^8(Z\times Z)\ \ \ \ (j=0,2,4,6,8)\ ,      \]
    where the $R_Z^k$'s are as in Proposition \ref{pfaf}.
    One proceeds to check that each $\Gamma_j$ is homologically trivial (this is easy, as one only needs to understand the action of $\iota$ on $\sym^\ast H^2_{tr}(Z)$). Then, using finite-dimensionality of $Z$, one finds that
     \[ (\Gamma_j)^{\circ N}=0\ \ \ \hbox{in}\ A^8(Z\times Z)\ ,\]
     for some $N\in\NN$. Using the fact that 
     \[  (\Gamma_j)_\ast(\Gamma_j)_\ast=(\Gamma_j)_\ast\colon\ \ A^8(Z)\ \to\ A^8(Z) \]
     (proof of Lemma \ref{OK}), we find that
     \[ (\Gamma_j)_\ast=0\colon\ \ A^8(Z)\ \to\ A^8(Z) \ ,\]
     which proves Theorem \ref{main2}.
      \end{proof}
  
  \begin{remark}\label{general} It makes sense to ask whether Theorem \ref{main2} can be extended to the general \llsvs{} eightfold. To obtain such an extension, the first problem is to construct a ``good'' \splitting{} of $A^8(Z)$ (in the sense that it is related to the filtration of \cite{V14}, and comes from a CK decomposition). The second problem is to show that $\iota$ acts in the expected way (this seems difficult without knowing $Z$ has finite-dimensional motive).
  \end{remark}

  \section{Some consequences}\label{section - Consequences}

\subsection{Constant cycle subvarieties}

The results of Sections \ref{ss:vois} and \ref{section - Main results} give interesting consequences on the existence of {\em constant cycle subvarieties\/} in some LLSvS variety $Z$, which are by definition subvarieties whose points are all rationally equivalent in $Z$ to a given $0$-cycle (see \cite{Huybrechts} and \cite{V14} for further details).

\begin{enumerate}
\item It follows from Corollary \ref{cor:split} that $A^8_{(0)}(Z)\cong\QQ$ and so there is a ``canonical $0$-cycle'' on $Z$. We remark that on the Fano variety of lines of any cubic fourfold $X$ there is a well-known constant cycle surface, constructed by Voisin in \cite[Lemma 2.2]{V12}: it is enough to choose $W=X\cap H$ to be a hyperplane section of $X$ with five nodes and to consider the Fano surface $F(W)$ of lines on $W$. It follows from \cite[Corollary 4.9]{V14} that the closure in $Z$ of the image via $\psi$ of $(F(W)\times F(W))\setminus I$ is again a constant cycle fourfold, which by construction is contained in the image $Z_H\subset Z$ via $u$ of the variety of twisted cubics contained in $W$.
 
 It is possible to be more precise: indeed,  for the eightfold $Z$ the canonical $0$-cycle is the degree $1$ generator of $\QQ[c_8(T_Z)]\subset A^8(Z)$, which is $\frac{1}{25650}c_8(T_Z)$ in $A^8(Z)$. 
 In view of Corollary \ref{cor:split}, any point $x\in S_4(Z)$ is such a degree $1$ generator (the locus $S_4(Z)$ is of dimension $4$, as proven by Voisin \cite[Corollary 4.9]{V14}).
 
 \item  Moreover, let $Z=Z(Y)$ be the \llsvs{} eightfold associated to a cyclic cubic fourfold $Y$ not containing a plane. As before, let $Z_H\subset Z$ denote the Lagrangian submanifold parametrizing twisted cubics contained in the hyperplane section $Y_H = Y\cap \{x_5=0\}$.
  Let us consider points
  \[ x\ \ \in\ Z_H\cap S_2(Z)\ .\]
  Since $x\in Z_H=\hbox{Fix}(g)$, Theorem \ref{main} implies that (the class of) $x$ is in $A^8_{(0)}(Z)\oplus A^8_{(6)}(Z)\oplus A^8_{(8)}(Z)$.
  Since $x\in S_2(Z)$, we know from Corollary \ref{cor:split} that (the class of) $x$ is in $\oplus_{j\le 2} A^8_{(2j)}(Z)$.
  It follows that $x\in A^8_{(0)}(Z)\cong\QQ$, i.e.\ irreducible components of $Z_H\cap S_2(Z)$ are constant cycle subvarieties.      
  
  Similarly, let $Z=Z(Y)$ be the \llsvs{} eightfold associated to a cyclic and Pfaffian cubic fourfold $Y$ not containing a plane. Theorems \ref{main} and \ref{main2}, combined with Corollary \ref{cor:split},
  imply that any point
   \[ x\ \ \in\ \hbox{Fix}(g)\cap \hbox{Fix}(\iota) \cap S_1(Z) \]
   has class in $A^8_{(0)}(Z)\cong\QQ$. 
\end{enumerate}

 \subsection{Intersection product}
 
 Recall 
 that the Chow groups of a quotient variety have an intersection product. In this subsection we study this product for certain quotients of the \llsvs{} eightfold $Z$.

 \begin{proof}[Proof of Corollary \ref{cor5}] The Pfaffian hypothesis ensures that the motive of $Z$ is isomorphic to the motive of the Hilbert scheme $S^{[4]}$, where $S$ is the associated K3 surface of degree $14$, as shown in Proposition \ref{K3iso}.
 In particular, it follows that the Chow ring of $Z$ has a bigrading given by an MCK decomposition (Theorem \ref{K3m}).

 Since the motive of $Z$ is isomorphic to the motive of the Hilbert scheme $S^{[4]}$, there is a ``hard Lefschetz type'' isomorphism
  \begin{equation}\label{hard}  \cdot h^3\colon\ \ \ A^2_{(2)}(Z)\ \xrightarrow{\cong}\ A^8_{(2)}(Z)\ ,\end{equation}
  for any ample divisor class $h\in A^1(Z)$ (see \cite[Corollary 3.2]{EPWcube}). Taking a $g$-invariant ample class $h$, this gives in particular an isomorphism
  \[  \cdot h^3\colon\ \ \ A^2_{(2)}(Z)^{\langle g \rangle}\ \xrightarrow{\cong}\ A^8_{(2)}(Z)^{\langle g \rangle}\ ,\] 
  and so Theorem \ref{main} implies that
  \[  A^2_{(2)}(Z)^{\langle g \rangle}=0\ .\]
  A similar reasoning also shows that
   \[ A^2(Z)^{\langle g \rangle}\ \subset\ A^2_{(0)}(Z)\ .\]
   Indeed, let $b\in A^2(Z)^{\langle g \rangle}$, and let $h\in A^1(Z)$ be once again a $g$-invariant ample class. Then $b\cdot h^3\in A^8(Z)^{\langle g \rangle}$ and so (in view of Theorem \ref{main}) we know that $b\in A^8_{(0)}(Z)$. Writing the decomposition $b=b_0+b_2$, where $b_j\in A^2_{(j)}(Z)$, and invoking the isomorphism (\ref{hard}), we see that $b_2=0$.
   
   The corollary is now readily proven. For instance, the image of the intersection map
   \[ \ima \bigl( A^4(Z)^{\langle g \rangle}\otimes A^2(Z)^{\langle g \rangle}\otimes A^2(Z)^{\langle g \rangle}\ \to\ A^8(Z)\bigr) \]   
   is contained in
   \[ \ima \left(  \left( \bigoplus_{j\le 4} A^4_{(j)}(Z)\right)\otimes A^2_{(0)}(Z)\otimes A^2_{(0)}(Z)\ \to\ A^8(Z)\right) \ \ \subset\  \bigoplus_{j\le 4} A^8_{(j)}(Z)\ .\]
   On the other hand, the image is obviously $\langle g \rangle$-invariant, and so
   \[     \ima \bigl( A^4(Z)^{\langle g \rangle}\otimes A^2(Z)^{\langle g \rangle}\otimes A^2(Z)^{\langle g \rangle}\ \to\ A^8(Z)\bigr)\ \ \subset\ A^8(Z)^{\langle g \rangle}\ \ \subset\ A^8_{(0)}(Z)\oplus A^8_{(6)}(Z) \oplus A^8_{(8)}(Z)\]   
   (where the second inclusion is Theorem \ref{main}). It follows that
     \[     \ima \bigl( A^4(Z)^{\langle g \rangle}\otimes A^2(Z)^{\langle g \rangle}\otimes A^2(Z)^{\langle g \rangle}\ \to\ A^8(Z)\bigr)\ \ \subset\ A^8_{(0)}(Z)\cong\QQ\ . \]   
     The other maps are treated similarly.
 \end{proof}
 
 We can also prove the following result concerning $1$-cycles.
 
  \begin{corollary}\label{cor5half} Let $Y\subset\PP^5$ be a Pfaffian cyclic cubic not containing a plane and let $Z=Z(Y)$ be the associated \llsvs{} eightfold. Let 
    $ Q:=Z/\langle g \rangle$ be the quotient, where
  $g\in\aut(Z)$ is the automorphism of order $3$ of Section \ref{sec: intro}. Then the image
   \[ \ima \bigl( A^2(Q)\otimes A^2(Q)\otimes A^2(Q)\otimes A^1(Q)\ \to\ A^7(Q)\bigr) \]
   injects into cohomology via the cycle class map.
                      \end{corollary}    
                      
              \begin{proof} Let $G:=\langle g \rangle$. As we have just seen (proof of Corollary \ref{cor5}), $Z$ has an MCK decomposition and $A^\ast(Z)$ has a bigrading. Since $A^i(Z)^{G}$ is equal to $A^i_{(0)}(Z)^{G}$ for $i=1,2$, there is an inclusion
  \[ \ima \bigl( A^2(Z)^G\otimes A^2(Z)^G\otimes A^2(Z)^G\otimes A^1(Z)^G\ \to\ A^7(Z)\bigr) \ \ \subset\ A^7_{(0)}(Z)\ .\]
 However, the subgroup $A^7_{(0)}(Z)\subset A^7(Z)$ is known to inject into $H^{14}(Z)$ via the cycle class map (cf.\ \cite[Introduction]{V6}).              
              \end{proof}

 \begin{remark}\label{pfaffian} Let $Y\subset\PP^5$ be any cyclic cubic fourfold (not necessarily Pfaffian) and let 
   $ Q:=Z(Y)/\langle g \rangle$
   be the quotient under the group generated by the order $3$ automorphism $g$. Conjecturally, the subring of $A^\ast(Q)$ generated by $A^1(Q), A^2(Q)$ (and the pushforwards of the Chern classes of $Z(Y)$) should inject into cohomology.
 
To prove this for cyclic non-Pfaffian cubics, we run into the problem that we don't know whether $Z(Y)$ has an MCK decomposition. To prove the full conjecture for Pfaffian cubics, we run into the problem that it is not known whether $A^i_{(0)}(Z(Y))$ injects into cohomology, except for $i=7,8$.
\end{remark}

 \begin{proof}[Proof of Corollary \ref{cor6}] Similar to Corollary \ref{cor5}. 
 \end{proof}

 \begin{remark} Let $Y\subset\PP^5$ be any (not necessarily cyclic, nor Pfaffian) cubic not containing a plane. Let $Z=Z(Y)$ be the associated \llsvs{} eightfold and let $\iota\in\aut(Z)$ be the anti-symplectic involution. One expects that the subring
   \[  \langle A^1(Z), A^2(Z)^\iota,  c_j(T_Z)\rangle\ \ \subset\ A^\ast(Z) \]
   injects into cohomology (this is the Beauville--Voisin conjecture for $Z$, extended by adding $A^2(Z)^\iota$ which is supposed to lie in $A^2_{(0)}(Z)$).
   It would be interesting to prove this for cases not covered by Corollary \ref{cor6}. 
  \end{remark}

 \vskip1cm
\begin{nonumberingt} The authors are grateful to the organisers of the ``Workshop on Algebraic Geometry'' held in Milano (November 2017), where the first idea of this project came up. We would also like to express our gratitude to Christian Lehn and to Manfred Lehn for various explanations concerning their work. Moreover, we would like to thank the Max Planck Institute of Mathematics in Bonn, where this paper was completed, for providing a warm working environment. Finally, the authors want to express their gratitude to the referee for his/her precious comments and corrections.

\end{nonumberingt}

\end{document}